\documentclass[12pt]{amsart}
\usepackage{amsmath,amssymb,amsfonts,amsthm,amstext,verbatim,url}
\usepackage{graphicx}
\RequirePackage[colorlinks,citecolor=blue,urlcolor=blue]{hyperref}

\theoremstyle{plain}

\newtheorem{theorem}{Theorem}

\numberwithin{equation}{section}
\numberwithin{theorem}{section}
\numberwithin{figure}{section}

\newcounter{mycount}

\newenvironment{numlist}{\begin{list}{\arabic{mycount}.}%
   {\usecounter{mycount}\labelwidth=1cm\itemsep 0pt}}{\end{list}}
\newenvironment{letlist}{\begin{list}{\rm(\alph{mycount})}%
   {\usecounter{mycount}\labelwidth=1cm\itemsep 0pt}}{\end{list}}

\newcommand\s{\sigma}
\newcommand\ul{\underline}
\newcommand\ol{\overline}
\newcommand\oo{\infty}
\newcommand\De{\Delta}
\newcommand\LL{{\mathbb L}}
\newcommand\HH{{\mathbb H}}
\newcommand\NN{{\mathbb N}}

\newcommand\sA{{\mathcal A}}

\newcommand\ZZ{{\mathbb Z}}

\newcommand\lf{\lambda_{\text{\rm f}}}

\newcommand\wt{\widetilde}
\newcommand\om{\omega}
\newcommand\es{\varnothing}
\newcommand\sm{\setminus}
\renewcommand\a{\alpha}

\newcommand\si{\sigma}
\renewcommand\b{\beta}
\newcommand\g{\gamma}
\newcommand\resp{respectively}

\newcommand\Aut{\text{\rm Aut}}

\newcommand\loopgraph{bridge graph}
\newcommand\LG{{\mathbb B}}

\begin{document}
\title[Bounds on connective constants]{Bounds on connective constants\\ of regular graphs}
\author[Grimmett]{Geoffrey R.\ Grimmett}
\address{Statistical Laboratory, Centre for
Mathematical Sciences, Cambridge University, Wilberforce Road,
Cambridge CB3 0WB, UK} 
\email{\{g.r.grimmett, z.li\}@statslab.cam.ac.uk}
\urladdr{\url{http://www.statslab.cam.ac.uk/~grg/}}
\urladdr{\url{http://www.statslab.cam.ac.uk/~zl296/}}

\author[Li]{Zhongyang Li}

\begin{abstract}
The \emph{connective constant} $\mu$
of a graph $G$ is the asymptotic growth
rate of the number of self-avoiding walks on $G$
from a given starting vertex.
Bounds are proved for the connective constant of an infinite,  
connected, $\De$-regular graph $G$. The main result is
that $\mu \ge \sqrt{\De-1}$ if $G$ is vertex-transitive and simple.
This inequality is proved subject to weaker conditions
under which it is sharp.

\end{abstract}

\date{23 October 2012, minor revision 6 April 2013}

\keywords{Self-avoiding walk, connective constant, cubic graph, regular graph, vertex-transitive graph, 
quasi-transitive graph, Cayley graph}
\subjclass[2010]{05C30, 82B20}

\maketitle

\section{Introduction}\label{sec:intro}
A self-avoiding walk (SAW) is a path on a graph that visits no vertex more than once. SAWs
were introduced as a model for long-chain polymers in chemistry (see \cite{f}),
and have since been studied intensively by mathematicians and physicists
interested in their critical behaviour (see \cite{ms}). 
If the underlying graph has a property of periodicity, the asymptotic behaviour
of the number of SAWs of length $n$
(starting at a given vertex) is exponential in $n$, 
with growth rate called the \emph{connective constant}
of the graph. The main purpose of this paper is
to explore upper and lower bounds for connective constants. 

The principal result of this paper is the following lower bound for
the connective constant $\mu$ of a $\De$-regular graph. The complementary upper bound
$\mu\le \De-1$ is very familiar.

\begin{theorem}\label{thm:cor0}
Let $\De \ge 2$, and let $G$ be an infinite, connected, 
$\De$-regular, vertex-transitive, simple graph. 
Then $\mu(G) \ge \sqrt{\De-1}$. 
\end{theorem}

The problem of counting SAWs is linked in two ways to the study of interacting disordered systems such as
percolation and Ising/Potts models. First, the numerical value of $\mu$ leads  to
bounds on critical points of such models (see \cite[eqns (1.12)--(1.13)]{G99} for percolation, and
hence Potts models via \cite[eqn (5.8)]{G-rcm}). Secondly, the SAW problem may be phrased
in terms of the SAW generating function; this has radius of convergence $1/\mu$, and the singularity is
believed to have power-law behaviour (for lattice-graphs such as $\ZZ^d$,
at least), see \cite{ms}. Thus, a lower bound for $\mu$
may be viewed as an upper bound for the critical point of a certain combinatorial problem. 
 
In Section 2, we introduce notation and definitions used
throughout this paper, and in Section 3 we prove a theorem 
concerning connective constants of general graphs.
Inequalities for the connective constant $\mu(G)$ of a $\De$-regular
graph $G$ are explored in Section \ref{sec:regular}, including a re-statement and discussion of
Theorem \ref{thm:cor0}.  It is shown 
at Theorem \ref{thm:upperbnd} that a quasi-transitive,
$\De$-regular graph $G$
satisfies $\mu(G) = \De-1$ if and only if $G$ is the $\De$-regular tree.
The proofs of results in Section \ref{sec:regular} are found in Section \ref{sec:regpf}.

There are two companion papers, \cite{GrL2,GrL3}. In \cite{GrL2}, we  use the Fisher transformation in
the context of SAWs on a cubic or partially cubic graph. In particular,
we calculate the connective constant of a 
certain lattice obtained from the hexagonal
lattice by applying the Fisher transformation at alternate vertices.
In \cite{GrL3}, we study strict inequalities between connective constants. 
It is shown (subject to minor conditions) that 
$\mu(G_2) < \mu(G_1)$ if either: (i)  $G_2$ is the quotient graph
of $G_1$ with respect to a non-trivial
normal subgroup of its automorphism group, or (ii) $G_1$
is obtained from $G_2$ by the addition of a further quasi-transitive family of edges. 

\section{Notation}\label{sec:notation}

All graphs in this paper will be assumed 
infinite, connected, and loopless  (a \emph{loop} is an edge both of whose endpoints are the same vertex).
In certain circumstances, they are permitted to have
\emph{multiple edges} (that is, two or more edges with the same endpoints). 
A graph $G=(V,E)$ is called \emph{simple} if it
has neither loops nor multiple edges. 
An edge $e$ with endpoints $u$, $v$ is written $e=\langle u,v \rangle$,
and two edges with the same endpoints are said to be \emph{parallel}.
If $\langle u,v \rangle \in E$, we call $u$ and $v$ \emph{adjacent}
and write $u \sim v$.
The \emph{degree} of vertex $v$ is the number of edges
incident to $v$, denoted $\deg(v)$. 
We assume that the vertex-degrees of a given graph $G$
are finite with supremum $\De$,
and shall often (but not always) assume $\De < \oo$.  
The \emph{graph-distance} between two vertices $u$, $v$ is the number of edges
in the shortest path from $u$ to $v$, denoted $d_G(u,v)$.

A \emph{walk} $w$ on $G$ is
an alternating sequence $v_0e_0v_1e_1\cdots e_{n-1} v_n$ of vertices $v_i$
and edges $e_i$ such that $e_i=\langle v_i, v_{i+1}\rangle$.
We write $|w|=n$ for the \emph{length} of $w$, that is, the number of edges in $w$.
The walk $w$ is called \emph{closed} if $v_0=v_n$. A \emph{cycle}
is a closed walk $w$ with $v_i\ne v_j$ for $1 \le i < j \le n$.
Thus, two parallel edges form a cycle of length $2$.

Let $n \in \{1,2,\dots\}\cup\{\oo\}$. 
An $n$-step self-avoiding walk (SAW) 
on $G$ is  a walk containing $n$ edges
that includes no vertex more than once.
Let $\s_n(v)$ be the number of $n$-step SAWs 
 starting at $v\in V$. We are interested here in the exponential growth rate
of $\s_n(v)$, and thus we define 
$$
\ul\mu(v)= \liminf_{n \to \oo} \s_n(v)^{1/n}, \qquad 
\ol\mu(v)= \limsup_{n \to \oo} \s_n(v)^{1/n}.
$$

There is a choice over the most useful way to define the connective constant
of an arbitrary graph. One such way is as the constant $\mu=\mu(G)$ given by
\begin{equation}
\label{connconst}
\mu(G) :=  \lim_{n\to\oo} \left(\sup_{v\in V} \s_n(v)^{1/n}\right).
\end{equation}
The limit in \eqref{connconst} exists for any graph by the
usual argument using subadditivity  
(see the proof of Theorem \ref{thm:infsup}(b)).
By Theorem \ref{jmh}, \eqref{connconst} provides an appropriate definition 
of the connective constant of a quasi-transitive graph.
The constant $\ol\mu$ of the forthcoming equation \eqref{connconst2} 
provides an alternative definition which is more suitable for a random
graph of the type studied by Lacoin \cite{Lac}.

It will be convenient to consider also SAWs starting at `mid-edges'. We identify the
edge $e$ with a point (also denoted $e$) placed at the middle of $e$, 
and then consider walks that start and end at 
these mid-edges. Such a  walk is 
called \emph{self-avoiding} if it visits no mid-edge or vertex more than once,
and its \emph{length} is the number of vertices visited.

The automorphism group of the graph $G=(V,E)$ is
denoted $\Aut(G)$. A subgroup $\sA \subseteq \Aut(G)$ is said to \emph{act
transitively} on $G$ 
if, for $v,w\in V$, there exists $\a \in \sA$ with $\a v=w$.
It acts \emph{quasi-transitively} if there exists a finite subset $W \subseteq V$ such that,
for $v \in V$ there exists $\a\in\sA$ such
that $\a v \in W$. The graph is 
called \emph{vertex-transitive} 
(\resp, \emph{quasi-transitive}) if $\Aut(G)$
acts transitively (\resp, quasi-transitively).

\section{Basic facts for general graphs}\label{sec:basic}

Hammersley \cite{jmhII}  proved that $\s_n(v)^{1/n} \to \mu$ for a class of
graphs including quasi-transitive graphs. 

\begin{theorem}\cite{jmhII}\label{jmh}
Let $G=(V,E)$ be an infinite, connected, quasi-transitive graph with finite vertex-degrees. Then
\begin{equation}\label{connconst0}
\lim_{n \to \oo} \s_n(v)^{1/n} = \mu, \qquad v \in V,
\end{equation}
where $\mu$ is given in \eqref{connconst}.
\end{theorem}

The picture is incomplete in the absence of quasi-transitivity. The following ancillary
result provides a partial connection between the $\ol\mu(v)$ and the $\ul\mu(v)$.
The proof appears later in this section.

\begin{theorem}\label{thm:infsup}
Let $G=(V,E)$ be an infinite, connected graph with finite vertex-degrees, 
and assume there exists $v \in V$ with $\ol \mu(v)< \oo$.
\begin{letlist}
\item  $\ol\mu(u) = \ol \mu(v)$ for all $u,v \in V$.
\item If $\ul \mu(v) = \mu$ for some $v \in V$, then $\ul\mu(v)=\mu$ for all
$v \in V$.
\end{letlist}
\end{theorem}

Part (a) may be found also in \cite[Lemma 2.1]{Lac}, which appeared
during the writing of this paper. It is in fact only a minor variation
of Hammersley's proof of Theorem \ref{jmh}. 
Subject to the conditions of Theorem \ref{thm:infsup},
let $\ol\mu=\ol\mu(G)$ be given by
\begin{equation}\label{connconst2}
\ol\mu := \ol\mu(v), \qquad v \in V.
\end{equation}
By Theorem \ref{jmh}, $\ol\mu(G) = \mu(G)$ for
any quasi-transitive graph $G$.
The constant $\ol\mu$ is a more appropriate definition of `connective constant' for 
a random graph, for which it can be the case that $\ol\mu < \mu$; see \cite{Lac}. 

Assume that the supremum vertex-degree $\De$ satisfies $\De<\oo$. It is elementary that
\begin{equation}\label{1}
1\le \ul \mu(v) \le  \mu\le \De-1, \qquad v \in V.
\end{equation}

\begin{figure}[htb]
 \centering
\includegraphics[width=0.8\textwidth]{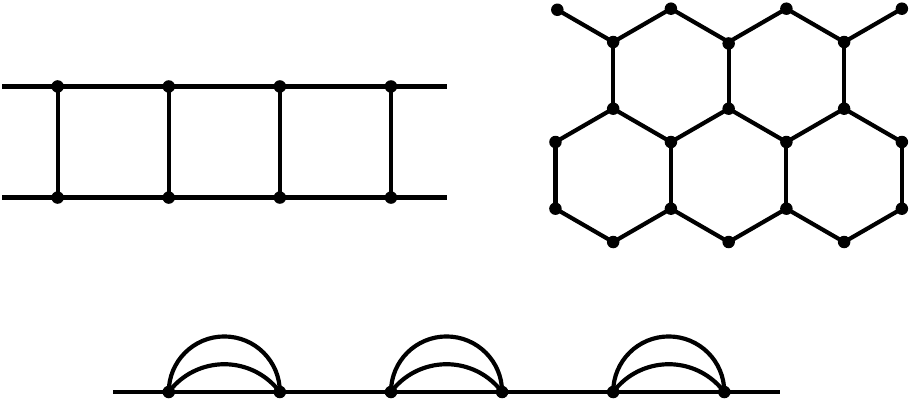}
  \caption{Three regular graphs: the  (doubly-infinite) ladder graph $\LL$;  
the hexagonal tiling $\HH$
of the plane; the \loopgraph\ $\LG_\De$ (with $\De=4$) obtained from $\ZZ$ by joining
every alternating pair of consecutive vertices by $\De-1$ parallel edges.}
  \label{fig:ladder-hex}
\end{figure}

The connective constant is known exactly for a limited class of 
quasi-transitive graphs, of which we mention
the \emph{ladder} $\LL$, the hexagonal lattice $\HH$, and the 
\emph{\loopgraph}  $\LG_\De$ with degree $\De\ge 2$ of Figure \ref{fig:ladder-hex}, for which
\begin{equation}\label{2}
\mu(\LL) = \tfrac12(\sqrt 5 + 1), \quad \mu(\HH) = \sqrt{2+\sqrt 2},
\quad \mu(\LG_\De) = \sqrt {\De-1}.
\end{equation}
See \cite[p.\ 184]{AJ90} and \cite{ds}
for the first two calculations. 
There is an extensive literature devoted to self-avoiding walks, 
including numerical upper and lower bounds for  connective constants, of which we mention
\cite{a04,bdgs,j04,ms}.

\begin{proof}[Proof of Theorem \ref{thm:infsup}]
We adapt and extend an argument of \cite{jmhII}. 
Let $u,v$ be neighbours
joined by an edge $e$.  Let $\pi$ be an $n$-step SAW from $u$.
Either $\pi$ visits $v$, or it does not. 
\begin{numlist}
\item If $\pi$ does not visit $v$, we prepend $e$ to obtain
an $(n+1)$-step SAW from $v$.
\item If $\pi$ visits $v$ after a number $m < n$ steps, we break $\pi$ after $m-1$ steps,
and prepend $e$ to the first subpath to obtain two SAWs from $v$: one of length $m$
and the other of length $n-m$.
\item If $\pi$ visits $v$ after $n$ steps, we remove the final edge and prepend $e$
to obtain an $n$-step SAW from $v$.
\end{numlist}
It follows that
\begin{equation}\label{15}
\s_{n}(u) \le \s_{n+1}(v) + \sum_{m=1}^{n-1} \s_m(v) \s_{n-m}(v)
+ \s_{n}(v).
\end{equation}

Suppose now that $\ol\mu(v)<\oo$, and let $\tau > \ol\mu(v)$. 
There exists
$C=C(\tau)<\oo$ such that
\begin{equation}\label{11}
\s_k(v) \le C \tau^k, \qquad k \ge 0.
\end{equation}
By \eqref{15},
\begin{equation}\label{12}
\s_{n}(u) \le C\tau^{n}(\tau + nC +1).
\end{equation}
Hence, $\ol\mu(u) \le \tau$ and therefore $\ol\mu(u) \le \ol\mu(v)$.
Part (a) follows since $G$ is connected and undirected.

Turning to part (b), we have in the usual way that
\begin{equation}\label{121}
\s_{m+n}(v) \le \s_m(v)\s_n,
\end{equation}
where $\s_n = \sup_{v\in V} \s_n(v)$.
Therefore, $\s_{m+n} \le \s_m \s_n$, whence the limit $\mu$ exists
in \eqref{connconst}.
We note in passing that
\begin{equation}\label{1215}
\s_n \ge \mu^n, \qquad n \ge 1.
\end{equation}
For $\tau>\mu$, there exists $C=C(\tau)<\oo$
such that
\begin{equation}\label{122}
\s_n \le C \tau^n, \qquad n \ge 0.
\end{equation}

Let $\tau > \mu$. By
\eqref{121}--\eqref{122}, there exists $C=C(\tau)<\oo$ such that
\begin{equation}\label{14}
\s_{i+j}(u) \le C\s_i(u) \tau^j, \qquad u \in V,\ i,j \ge 0.
\end{equation} 
Set $n=2k$ in \eqref{15}, and break the sum
into two parts depending on whether or not $m \le k$.
By \eqref{15} and \eqref{14},
\begin{equation}\label{13}
\s_{2k}(u) \le C\s_k(v) (\tau^{k+1} + 2kC \tau^{k} + \tau^k).
\end{equation} 
Therefore, $\ul\mu(u)^2 \le \ul\mu(v)\tau$, so that $\ul\mu(u)^2 \le \ul\mu(v)\mu$.
Assume that $u$ satisfies $\ul\mu(u) = \mu$. Then $\ul\mu(v) = \mu$, and
the claim follows by iteration.
\end{proof}

\section{Connective constants of regular graphs}\label{sec:regular}

The graph $G$ is \emph{regular} (or $\De$-\emph{regular})
if every vertex has the same degree $\De$. A $3$-regular graph
is called \emph{cubic}. In this
section, we investigate bounds for the connective constants of
infinite regular graphs. The optimal universal  lower bound,
even restricted to quasi-transitive graphs,  is of course the trivial  bound 
$\mu \ge 1$. This is achieved when $\De=3,4$ by the graphs of
Figure \ref{fig:regular3/4}, and by similar constructions for $\De \ge 5$.
Improved bounds may be proved when $G$ is assumed vertex-transitive.

\begin{figure}[htb]
 \centering
    \includegraphics[width=0.4\textwidth]{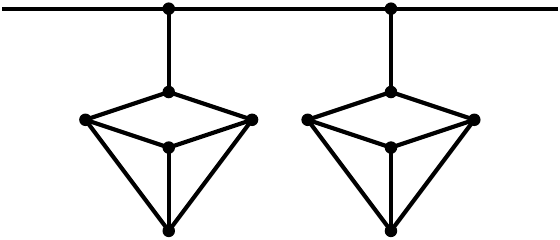}
\quad\includegraphics[width=0.4\textwidth]{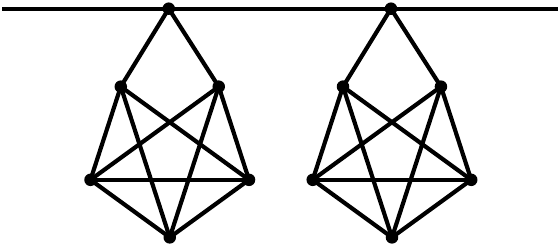}
  \caption{An infinite line may be decorated in order to obtain
regular graphs of degree $3$ and $4$. Similar contructions yield 
regular graphs with arbitrary degree $\De$ and connective constant $1$.}
  \label{fig:regular3/4}
\end{figure}

The main result of this paper, Theorem \ref{thm:cor0} is included in the following theorem,
of which the upper bound on $\mu(G)$ is already well known.

\goodbreak
\begin{theorem}\label{thm:cor1}
Let $\De \ge 2$, and let $G$ be an infinite, connected, 
$\De$-regular, vertex-transitive graph. 
We have $\mu(G) \le \De-1$, and in addition $\mu(G)
\ge \sqrt{\De-1}$  if either
\begin{letlist}
\item $G$ is simple, or
\item $G$ is non-simple and $\De \le 4$.
\end{letlist} 
\end{theorem}

Part (a) answers a question posed by Itai Benjamini (personal communication).
We ask whether the lower bound is \emph{strict} for \emph{simple} graphs,
and whether part (b) may be extended to larger values of $\De$.
Proofs of theorems in this section are found in Section \ref{sec:regpf}.

For a graph satisfying the initial conditions of Theorem \ref{thm:cor1}, 
we have by Theorem \ref{jmh}
that $\ul\mu(v) = \ol \mu(v)=\mu$
for all $v \in V$.
The Cayley graph (see \cite{bab95})
of an infinite group with finitely many generators 
satisfies the hypothesis of Theorem \ref{thm:cor1}(a).  
If the assumption of vertex-transitivity is weakened to
quasi-transitivity, the best lower bound is $\mu \ge 1$,
as illustrated in Figure \ref{fig:regular3/4}.

The upper bound of Theorem \ref{thm:cor1} is an equality for 
the $\De$-regular tree $T_\De$, but is strict for non-trees,
even within the larger class of quasi-transitive graphs.
We prove the slightly more general fact following,
thereby extending an earlier result of Bode \cite[Sect.\ 2.2]{bode} for 
quotients of free groups.

\begin{theorem}\label{thm:upperbnd}
Let $G=(V,E)$ be an infinite, connected, quasi-transitive
graph (possibly with multiple edges), and let $\De \ge 3$. We have that $\mu(G)<\De-1$ if either
\begin{letlist}
\item $G$ is $\De$-regular and contains a cycle, or 
\item $\deg(v) \le \De$ for all $v \in V$, and there exists $w \in V$ with $\deg(w)\le \De-1$.
\end{letlist}
\end{theorem}

It is a natural problem to decide when
the connective constant of a graph decreases strictly as further cycles are added.
Theorem \ref{thm:upperbnd} is a  step in this direction. When
the graphs are required to be regular, this question may be phrased in terms
of graphs and quotient graphs, and it is treated in \cite{GrL3}.

We shall deduce Theorem \ref{thm:cor1} from the stronger Theorem \ref{thm:main1} following.
The latter assumes a certain condition which we introduce next.
This condition
plays a role in excluding the graphs of Figure \ref{fig:regular3/4}. 
It is technical, 
but is satisfied by a variety of graphs of interest.

Let $G=(V,E)$ be an infinite, connected, $\De$-regular graph, 
possibly with multiple edges.  For distinct edges $e,e'\in E$ with a common vertex $w \in V$, 
a SAW is said to \emph{traverse} the triple $ewe'$ if it contains the mid-edge
$e$ followed  consecutively by the vertex $w$ and the mid-edge $e'$.   For $v \in V$,
let $I(v)$ be the set of infinite SAWs from $v$, and $I(e)$ the
corresponding set starting at the mid-edge of $e \in E$. 
Let $\pi\in I(v)$, let $ewe'$ be a triple traversed by $\pi$, and write $\pi_w$ for the 
finite subwalk of $\pi$ between $v$ and $w$. Let $e'' \ne e',e''$ be an edge incident to $w$.
We colour $e''$ \emph{blue} if there exists $\pi''\in I(v)$
that follows $\pi_w$ to $w$ and then takes edge $e''$, and 
we colour $e''$ \emph{red} otherwise. Let $R_{\pi,w} = \{e_j: j=1,2,\dots,r\}$
be the set of red edges corresponding to the pair $(\pi,w)$.

We make two notes. First, an edge of the form $\langle u,w\rangle$ with $u \in \pi_w$ can
be red when seen from $w$ and blue when seen from $u$. 
Thus, correctly speaking, colour is a property of a \emph{directed}
edge. We shall take care over this when necessary.
Secondly, suppose there is a group of
 two or more \emph{parallel} edges $e''=\langle w, w'\rangle$
with $w \in \pi$, $e'' \notin \pi$. Then all such edges have the same colour. 
They are all blue if and only if 
there exists $\pi''\in I(v)$
that follows $\pi_w$ to $w$ and then takes one of these edges.

The vertex $v \in V$ is
said to \emph{satisfy condition} $\Pi_v$ if, for all $\pi\in I(v)$
and all triples $ewe'$ traversed by $\pi$, there exists
a set $ F(\pi, w) = \{f_j = \langle x_j, y_j\rangle : j=1,2,\dots, |R_{\pi,w}|\}$ 
of distinct edges of $G$ such that,
for $1\le j \le |R_{\pi,w}|$, 
\begin{letlist}
\item $y_j \in \pi_w$, $y_j \ne w$,
\item there exists a SAW from $w$ to $x_j$ with first edge $e_j$,
 that is vertex-disjoint 
from $\pi_w$ except at its starting vertex $w$.
\end{letlist}
The graph $G$ is said to \emph{satisfy condition} $\Pi$ if
every vertex $v$ satisfies condition $\Pi_v$. 
The set $F(\pi,w)$ is permitted to contain parallel edges.
By reversing the SAWs in (b) above, we see that every edge 
$\langle x_j,y_j\rangle \in F(\pi,w)$ is blue
when seen from $y_j$. 
Note that $F(\pi,v)=\es$ for $\pi\in I(v)$.

\begin{figure}[htb]
 \centering
    \includegraphics[width=0.8\textwidth]{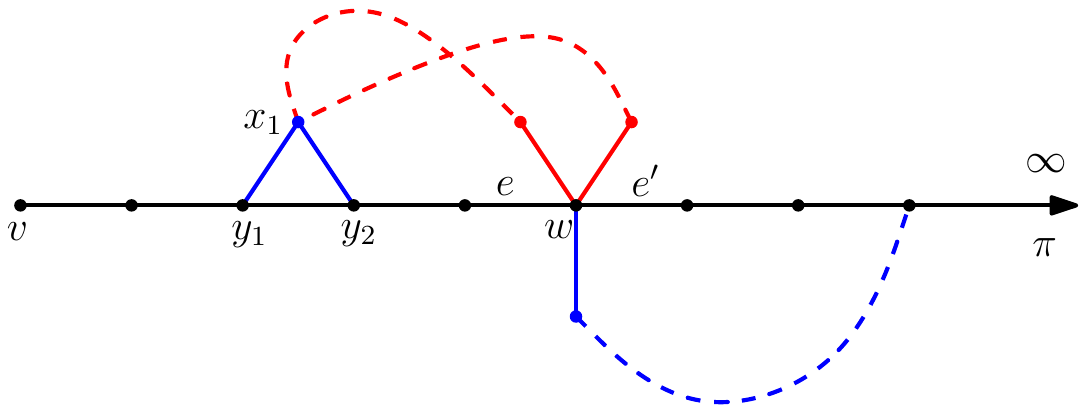}
  \caption{An illustration of the condition $\Pi_v$
with $\De=5$ and $x_1=x_2$.}
  \label{fig:condition}
\end{figure}

Condition $\Pi_v$ may be expressed in a simpler form for cubic graphs (with $\De=3$). 
In this case, 
for each pair $(\pi,w)$ there exists at most one red edge. Therefore, $\Pi_v$ is equivalent 
to the following: for every $\pi\in I(v)$ and every triple $ewe'$ traversed by $\pi$,
there exists $\pi''\in I(e'')$  beginning $e''w''$,
where $e'' = \langle w, w''\rangle$ is the third edge incident
with $w$. 
It is thus sufficient for a cubic graph $G$ that every vertex lies in some
doubly-infinite self-avoiding walk of $G$. 

\begin{theorem}\label{thm:main1}
Let $\De\ge 2$, and let $G=(V,E)$ be an infinite, connected $\De$-regular graph. If $v \in V$ satisfies
condition $\Pi_v$, we have $\ul\mu(v) \ge \sqrt{\De-1}$.
The \loopgraph\ $\LG_\De$ satisfies condition $\Pi$,
and $\ul\mu(v) = \mu = \sqrt{\De-1}$ for
all vertices $v$.
\end{theorem}

It is trivial that the $\De$-regular tree $T_\De$ satisfies condition $\Pi$ and
has connective constant $\De-1$, and it was noted in \eqref{1} that $\De-1$ is an upper
bound for connective constants of $\De$-regular graphs. 
Let $\De \ge 2$ and $\sqrt{\De-1} \le \rho \le \De-1$. By replacing the edges
of $T_\De$ by finite segments of the \loopgraph\ $\LG_\De$, one may construct graphs 
satisfying condition $\Pi$ with connective constant
$\rho$. 
Therefore, the set of connective constants of infinite, connected, $\De$-regular
graphs satisfying condition $\Pi$ is exactly the closed interval $[\sqrt{\De-1}, \De-1]$.

\section{Proofs of Theorems \ref{thm:cor1}--\ref{thm:main1}}\label{sec:regpf}

\begin{proof}[Proof of Theorem \ref{thm:main1}]

Let $G$ satisfy the given conditions. A finite SAW is called \emph{extendable}
if it is the starting sequence of some infinite SAW.  Let $v \in V$ satisfy condition $\Pi_v$, and 
let $\wt \s_n$ be the number of extendable $n$-step SAWs from $v$.
We claim that
\begin{equation}\label{3}
\liminf_{n\rightarrow\infty} \wt \s_n ^{1/n}\geq \sqrt{\De-1},
\end{equation}
from which the inequality of the theorem follows. The claim is trivial when $\De=2$,
and we assume henceforth that $\De \ge 3$.
(Since this paper was written, the growth rate for extendable SAWs has been considered in \cite{GHLP}.)

Let $\pi = v_0 e_0v_1\cdots e_{2n-1}v_{2n}$ 
be an extendable $2n$-step SAW from $v_0=v$,
and, for convenience,  augment $\pi$ with a mid-edge $e_{-1}$ ($\ne e_0$) incident to $v_0$. 
Thus, $\pi$ traverses the triples $e_{s-1}v_se_s$ for $0 \le s < 2n$.
Let $r_s$ and $b_s$ be the numbers of red and blue edges, 
respectively, seen from $v_s$, so  that
\begin{equation}\label{4}
r_s+b_s = \De -2,  \qquad 0 \le s < 2n.
\end{equation}
We claim that
\begin{equation}
\label{5}
\sum_{s=0}^{2n-1} b_s \ge  n (\De-2),
\end{equation}
and the proof of this follows.

For $0\le s < 2n$, let $F_s = F(\pi,v_s)$, and recall
that $F(\pi,v_0)=\es$. 
We claim that $F_s \cap F_t = \es$ for $s \ne t$.
Suppose on the contrary that $0 \le s < t < 2n$ and $f \in F_s \cap F_t$
for some edge $f = \langle x,y\rangle$ with $y = v_u$ and $u<s$. 
See Figure \ref{fig:distinct}. There exists a SAW $\om_s$ from $v_s$ to $x$ such that:
(i) the first edge of $\om_s$, denoted $e_s$, is red,  and
(ii) $\om_s$ is vertex-disjoint from $\pi_{v_s}$ except at $v_s$.
Similarly, there exists a 
SAW $\om_t$ from $v_t$ to $x$ whose first edge $e_t$ is red,
and which is vertex-disjoint from $\pi_t$ except at $v_t$. 
Let $z$ be the earliest vertex of $\om_s$ lying in $\om_t$.
Consider the infinite SAW $\om'$ that starts at $v_s$, takes edge $e_s$, follows $\om_s$ to $z$,
then $\om_t$ and  $e_t$ backwards to $v_t$, and then follows $\pi \sm \pi_{v_t}$. Thus, $\om'$ is 
an infinite SAW starting with $v_se_s$ that is vertex-disjoint from $\pi_{v_s}$ except at $v_s$. 
This contradicts the colour of $e_s$ (seen from $v_s$), 
and we deduce that  $F_s\cap F_t = \es$
as claimed.

\begin{figure}[htb]
 \centering
    \includegraphics[width=0.8\textwidth]{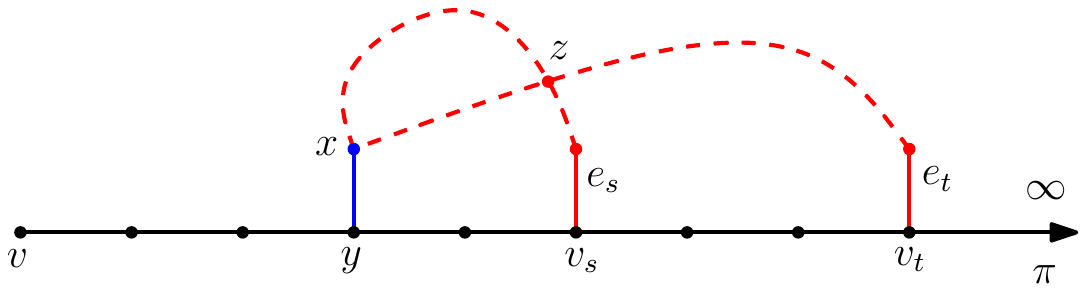}
  \caption{An illustration of the proof that $F_s \cap F_t = \es$.}
  \label{fig:distinct}
\end{figure}

Now,
$$
\sum_{s=0}^{2n-1} b_s \ge \sum_{s=0}^{2n-1} |F_s| =\sum_{s=0}^{2n-1} r_s.
$$  
The total number of blue/red edges is $2n(\De-2)$, and \eqref{5} follows.

We show next that \eqref{5} implies the claim of the theorem.
A \emph{branch} of $\pi$ (with \emph{root} $x$)
is an edge $e=\langle x,y\rangle$ such that $x \in \pi$, 
$x \ne v_{2n}$, $e \notin\pi$,
and the path $x e y$ lies in some $\pi'\in I(v)$.
A set of branches with the same root is called a \emph{group}.
By \eqref{5}, 
$\pi$ has at least $n(\De-2)$ branches, namely the blue edges. 
Each of these branches gives rise to a further extendable $2n$-step SAW from $v$, 
and similarly every such SAW has at least
$n(\De-2)$ such branches.  
We wish to understand how to group the branches on these walks
in order to minimize the total number of ensuing $2n$-step SAWs.

Let $B=\a(\De-2)+\b$ with $0\le \b < \De-2$, and let
$\pi$ be an extendable $2n$-step SAW from $v$, as above. 
Suppose 
that there are exactly $B$ branches along every ensuing 
(extendable) $2n$-step SAW $\pi'$,
and that no vertex of such a $\pi'$ is the endvertex of more than $\De-2$ branches. If the 
group of  branches of $\pi$ closest to $v$ has size $\b$, and all other groups have size $\De-2$, 
the number of ensuing $2n$-step SAWs is 
$g(B) := (\b+1)(\De-1)^\a$.
It will suffice  to show that, if every such $2n$-step SAW has exactly $B$ branches, 
then the total number of SAWs is at least $g(B)$. We  prove this by induction on $B$.

The claim is trivially true when $B=1$, since both numbers then equal $2$.
Suppose $B_0 \ge 1$ is such that the claim is true for $B \le B_0$,
and consider the case $B=B_0+1$.
Let $B=\a(\De-2)+\b$ as above. Pick $\pi$ as above, and suppose the first
group of branches along $\pi$ has size $\g$ for some
$\g$ satisfying $1\le \g\le\De-2$. 

There are two cases depending on whether or not $\g\le\b$. 
Assume first that $\g \le \b$.
The number of SAWs is at least $(\g+1)g(B-\g)$, which satisfies
$$
(\g+1) g(B-\g) = (\g+1)(\b-\g+1)(\De-1)^\a  \ge g(B),
$$
as required. In the second case ($\g > \b$), the corresponding inequality
$$
(\g+1)g(B-\g) = (\g+1)(\De-2 + \b-\g+1)(\De-1)^{\a-1} \ge g(B)
$$
is  quickly checked (since the middle expression is an upwards pointing 
quadratic in $\gamma$, it suffices
to check the two extremal cases $\gamma = \b+1, \De-2$), and the induction is complete. 

With $B \ge n(\De-2)$, we find that $\wt\s_{2n} \ge (\De-1)^n$, whence 
\begin{equation*}
\liminf_{n\to\oo} \wt\s_{2n}^{1/n} \ge  \De-1,
\end{equation*}
and \eqref{3} follows since $\wt\s_k$ is non-decreasing in $k$.

Let $\De \ge 2$. It is easily seen that the \loopgraph\ $\LG_\De$ satisfies
condition $\Pi$ and has connective constant $\sqrt{\De-1}$, as in \eqref{2}
\end{proof}

\begin{proof}[Proof of Theorem \ref{thm:cor1}]
The upper bound for $\mu(G)$ is as in \eqref{1}.

\noindent
(a) Let $G=(V,E)$ satisfy the given conditions.
We claim that, for $v \in V$, there exist $\De$ edge-disjoint
infinite SAWs from $v$. 
It follows that $G$ satisfies condition $\Pi$, and hence
part (a).

The claim is proved as follows. Let $\lf=\lf(G)$ be the least number of edges whose removal
disconnects  $G$ into components at least one of which is finite.
By \cite[Lemma 3.3]{BW} (see also \cite[Chap.\ 12, Prob.\ 14]{Lov}),
we have that $\lf = \De$. It is a consequence of
Menger's theorem that there exist $\lf$ edge-disjoint
infinite SAWs from $v$. A sketch of this presumably standard fact follows.
Let $n \ge 1$, and let $B_n$ be the graph obtained from
$G$ by identifying all vertices distance $n+1$ or more from $v$.
The identified vertex is denoted $\partial B_n$. 
Since $v$ has degree $\De$ and $\lf=\De$, the minimum
number of edges whose removal disconnects $v$ from $\partial B_n$ is
$\De$. By Menger's theorem (see \cite[Sect.\ 3.3]{Die}),
there exist $\De$ edge-disjoint SAWs from $v$ to $\partial B_n$.
Therefore, for all $n$, $G$ contains $\De$ edge-disjoint $n$-step SAWs from $v$.
Since $G$ is locally finite, this implies the above claim.

\noindent (b)
When $\De=2$, $G$ is simple. If $\De=3$
and $G$ is non-simple, it is immediate  that every $v$ has property $\Pi_v$,
and the claim follows by Theorem \ref{thm:main1}.
Suppose $\De=4$. There are three types of non-simple graph,
depending on the groupings of the parallel edges incident to a given vertex.
By consideration of these types, we see that only one type merits a detailed
argument, namely that in which each vertex is adjacent to exactly three other vertices,
and we restrict ourselves henceforth to this case.

Two paths from $w \in V$ are called \emph{vertex-disjoint} if
$w$ is their unique common vertex.
Let $\pi\in I(v)$, with vertex-sequence $(v, v_1,v_2,\dots)$. Then $v_n$ is the
endpoint of two vertex-disjoint SAWs of respective lengths $n$ and $\oo$.
By vertex-transitivity, for every $n \ge 1$ and $w \in V$, $w$ is the
endpoint of two vertex-disjoint SAWs of respective lengths $n$ and $\oo$.
Since $G$ is locally finite, every $w \in V$ is the endpoint of two vertex-disjoint infinite SAWs.  
We write the last statement as $v \Longrightarrow \oo$.

Let $\pi \in I(v)$, $w=v_k$ with $k \ge 1$, 
and consider the triple $ewe'$ traversed by $\pi$. By assumption,
$w$ has three neighbours $w_1$, $w_2$, $w_3$ in $G$, labelled in such a way that
$w_1 = v_{k-1}$ and $w_2= v_{k+1}$. 
For some $i$, there are two parallel edges of the form $\langle w,w_i\rangle$,
as illustrated in Figure \ref{fig:four}. 

\begin{figure}[htbp]
  \centering
\includegraphics[width=0.9\textwidth]{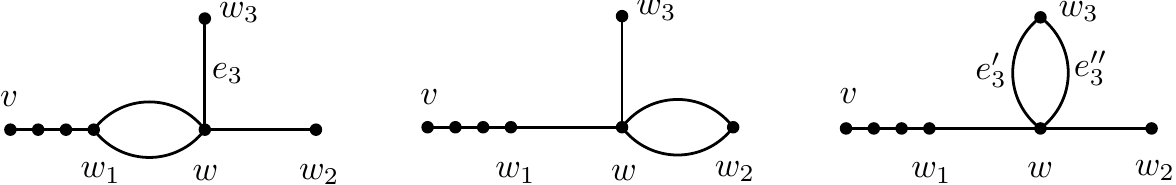}
   \caption{The three cases in the proof Theorem \ref{thm:cor1}(b).}
\label{fig:four}
\end{figure}

There are several cases to consider. If $w_3 \in \pi$, say $w_3=v_M$,
any edge $\langle w, w_3\rangle$ is red if
$M<k$ and blue otherwise. The situation is more interesting if $w_3 \notin \pi$, and we assume this henceforth.

Consider the first case in  Figure \ref{fig:four} (the second case is similar).
The edge $e_1=\langle w, w_1\rangle$ not in $\pi$ is red
(seen from $w$), and contributes itself to 
the set $F(\pi,w)$. Suppose $e_3=\langle w,w_3\rangle$ is red (if it is blue, there is nothing
to prove). Since $w_3 \Longrightarrow \oo$, there exists an infinite SAW $\nu$ from $w_3$
not using $e_3$. Since $e_3$ is assumed red, 
there exists a first vertex $z$ of $\nu$ lying in $\pi$, and furthermore $z =v_K$
for some $K < k$. We add to $F(\pi,w)$ the last edge of $\nu$ before $z$.

Consider the third case in Figure \ref{fig:four}, with parallel edges
$e_3'=\langle w, w_3\rangle$, $e_3''=\langle w, w_3\rangle$.
Since $e_3'$ and $e_3''$ have the same colour we may restrict ourselves to the case when both are red.
Since $w_3 \Longrightarrow \oo$,  there exists an infinite SAW $\nu$ from $w_3$
using neither $e_3'$ nor $e_3''$. Since $e_3'$ and $e_3''$ are assumed red,
there exists an earliest vertex $v_K$ of $\nu$ lying
in $\pi$, with $K<k$. Write
$\nu'$ for the sub-path of $\nu$ that terminates at $v_K$, and $f(\nu')$ 
for the final edge of $\nu'$. 

Let $g=\langle w_3,x\rangle$ 
be the edge incident to $w_3$ other than $e_3'$, $e_3''$, and the first edge of $\nu$.
We construct a path $\rho$ from $w_3$ with first edge $g$, as follows.
Suppose $\rho$ has been found up to some vertex $z$.
\begin{numlist}
\item If $z$ has been visited earlier by $\rho$, we exit $z$ along the 
unique edge not previously traversed by $\rho$. Such an edge exists since 
$G$ is $4$-regular.
\item If $z$ lies in $\nu'$, we exit $z$ along the unique edge
lying in neither $\nu'$ nor the prior part of $\rho$.
\item If $z \in \pi$, say $z=v_L$, we stop the construction, and write $f(\rho)$ 
for the final edge traversed.
\end{numlist}
Since $e_3'$ and $e_3''$ are red, Case 3 occurs for some $L<k$.
By construction, $\rho$ and $\nu'$ are edge-disjoint, whence
$f(\rho) \ne f(\nu')$.
Corresponding to the two red edges $e_3'$, $e_3''$, we
have the required set $F(\pi,w) = \{f(\nu'),f(\rho)\}$.

In conclusion, every $v \in V$ has property $\Pi_v$,
and the claim follows by Theorem \ref{thm:main1}.
\end{proof}

\begin{proof}[Proof of Theorem \ref{thm:upperbnd}]
Let $u\in V$ and let $e \in E$ be incident to $u$.
Let $\si_n(u,e)$ be the number
of $n$-step SAWs from $u$ that do not traverse $e$.
We shall prove, subject to either (a) or (b), that there exists $N =N(G) \ge 3$,  such that
\begin{equation}\label{g456}
\si_N(u,e) \le (\De-1)^N-1 \qquad\mbox{for all such pairs $u$, $e$}.
\end{equation}

Assume first that (a) holds.
By quasi-transitivity, there exist $M, l \in \NN$ and a cycle $\rho$ of length $l$
such that, for $v \in V$,
there exist $w\in V$ and $\a\in \Aut(G)$ such that $d_G(v,w)<M$ and 
\begin{equation}\label{g458}
w \in\a(\rho).
\end{equation}

Let $C(u,e)$ be the subset of $V$ reachable from $u$ along paths not using $e$.
If $|C(u,e)|<\oo$, then $\si_n(u,e)=0$ for large $n$,
whence \eqref{g456} holds for all $N$ larger than some $N(u,e)$.
Assume that  $|C(u,e)|=\oo$. 
Let $\pi=(\pi_0,\pi_1,\dots)$ be an infinite SAW from $u$ not using $e$.
This walk has a first vertex, $\pi_R$ say, lying at
distance $4M$ from $u$. By the definition of $M$, there exists $k=k(u,e,M)$
satisfying $R-M \le k \le R+M$, and a $k$-step SAW $\pi'$ from $u$ not using $e$,
such that $\pi'$ has 
final endpoint $w'$ lying in $\a'(\rho)$ for some $\a'\in\Aut(G)$. 
We may represent the set of SAWs from $u$, not using $e$,
as a subtree of the rooted tree of degree $\De$ (excepting 
the root, which has degree $\De-1$). By counting the number of paths in that tree,
we deduce that, for $N \ge N_0 := k+l+1$,  the number of
such $N$-step walks is no
greater than $(\De-1)^N-1$. 

Since $G$ is quasi-transitive, $N_0<\oo$ may be picked
uniformly in $u$, $e$. Inequality \eqref{g456} is proved in case (a).

If (b) holds, condition \eqref{g458} is replaced by $\deg(w) \le \De-1$,
and the conclusion above is valid for $N \ge k+2$.
For both cases (a) and (b), \eqref{g456} is proved.

By considering the last edge traversed by a $(k-1)N$-step
SAW from $v$, we have that
$$
\si_{kN}(v) \le \si_{(k-1)N}(v) \bigl[(\De-1)^N-1\bigr], 
\qquad k \ge 2,
$$
and, furthermore, $\si_N(v) \le \De[(\De-1)^{N}-1]$. 
Therefore,
$$
\mu = \lim_{n\to\oo} \s_n(v)^{1/n} \le \bigl[(\De-1)^N-1\bigr]^{1/N}
< \De-1,
$$
and the theorem is proved. 
\end{proof}

\section*{Acknowledgements} 
This work was supported in part by the Engineering
and Physical Sciences Research Council under grant EP/103372X/1.
It was completed during a visit by GRG to the Theory Group
at Microsoft Research. 
The authors thank Hubert Lacoin for 
indicating an error in an earlier version of this paper.

\bibliography{saw2-final2}

\providecommand{\bysame}{\leavevmode\hbox to3em{\hrulefill}\thinspace}
\providecommand{\MR}{\relax\ifhmode\unskip\space\fi MR }
\providecommand{\MRhref}[2]{%
  \href{http://www.ams.org/mathscinet-getitem?mr=#1}{#2}
}
\providecommand{\href}[2]{#2}
\begin{thebibliography}{10}

\bibitem{a04}
S.~E. Alm, \emph{Upper and lower bounds for the connective constants of
  self-avoiding walks on the {A}rchimedean and {L}aves lattices}, J. Phys. A:
  Math. Gen. \textbf{38} (2005), 2055--2080.

\bibitem{AJ90}
S.~E. Alm and S.~Janson, \emph{Random self-avoiding walks on one-dimensional
  lattices}, Commun. Statist. Stoch. Mod. \textbf{6} (1990), 169--212.

\bibitem{bab95}
L.~Babai, \emph{Automorphism groups, isomorphism, reconstruction}, Handbook of
  Combinatorics, vol.~II, Elsevier, Amsterdam, 1995, pp.~1447--1540.

\bibitem{BW}
L.~Babai and M.~E. Watkins, \emph{Connectivity of infinite graphs having a
  transitive torsion group action}, Arch. Math. \textbf{34} (1980), 90--96.

\bibitem{bdgs}
R.~Bauerschmidt, H.~Duminil-Copin, J.~Goodman, and G.~Slade, \emph{Lectures on
  self-avoiding-walks}, Probability and Statistical Physics in Two and More
  Dimensions (D.~Ellwood, C.~M. Newman, V.~Sidoravicius, and W.~Werner, eds.),
  Clay Mathematics Institute Proceedings, vol.~15, CMI/AMS publication, 2012,
  pp.~395--476.

\bibitem{bode}
J.~S. Bode, \emph{Isoperimetric constants and self-avoiding walks and polygons
  on hyperbolic {C}oxeter groups}, Ph.D. thesis, Cornell University, 2007,
  \url{http://dspace.library.cornell.edu/bitstream/1813/7522/1/thesis.pdf}.

\bibitem{Die}
R.~Diestel, \emph{{Graph Theory}}, Springer, Berlin, 2010.

\bibitem{ds}
H.~Duminil-Copin and S.~Smirnov, \emph{The connective constant of the honeycomb
  lattice equals $\sqrt{2+\sqrt 2}$}, Ann. Math. \textbf{175} (2012),
  1653--1665.

\bibitem{f}
P.~Flory, \emph{{Principles of Polymer Chemistry}}, Cornell University Press,
  1953.

\bibitem{G99}
G.~R. Grimmett, \emph{Percolation}, 2nd ed., Springer, Berlin, 1999.

\bibitem{G-rcm}
\bysame, \emph{{The Random-Cluster Model}}, Springer, Berlin, 2006.

\bibitem{GHLP}
G.~R. Grimmett, A.~E. Holroyd, and Y.~Peres, \emph{Extendable self-avoiding
  walks},  (2013), in preparation.

\bibitem{GrL2}
G.~R. Grimmett and Z.~Li, \emph{Self-avoiding walks and the {F}isher
  transformation},  (2012), \url{http://arxiv.org/abs/1208.5109}.

\bibitem{GrL3}
\bysame, \emph{Strict inequalities for connective constants of transitive
  graphs},  (2013), \url{http://arxiv.org/abs/1301.3091}.

\bibitem{jmhII}
J.~M. Hammersley, \emph{Percolation processes {II. T}he connective constant},
  Proc. Camb. Phil. Soc. \textbf{53} (1957), 642--645.

\bibitem{j04}
I.~Jensen, \emph{Improved lower bounds on the connective constants for
  two-dimensional self-avoiding walks}, J. Phys. A: Math. Gen. \textbf{37}
  (2004), 11521--11529.

\bibitem{Lac}
H.~Lacoin, \emph{Non-coincidence of quenched and annealed connective constants
  on the supercritical planar percolation cluster},  (2012),
  \url{http://arxiv.org/abs/1203.6051}.

\bibitem{Lov}
L.~Lov{\'a}sz, \emph{{Combinatorial Problems and Exercises}}, North-Holland
  Publishing Co., Amsterdam, 1979.

\bibitem{ms}
N.~Madras and G.~Slade, \emph{{Self-Avoiding Walks}}, Birkh\"auser, Boston,
  1993.

\end{thebibliography}
\bibliographystyle{amsplain}

\end{document}